\documentclass{article}

\usepackage[english]{babel}
\usepackage[utf8]{inputenc}
\usepackage{amsmath}
\usepackage{amsthm}
\usepackage{amssymb}
\usepackage{graphicx}
\usepackage[colorinlistoftodos]{todonotes}
\usepackage{hyperref} 
\usepackage{color}
\usepackage{verbatim}

\newtheorem{theorem}{Theorem}
\newtheorem{lemma}[theorem]{Lemma}

\theoremstyle{definition}

\title{Who Is Guilty?}
\author{Benjamin Chen, Ezra Erives, Leon Fan,\\
Michael Gerovitch, Jonathan Hsu, Tanya Khovanova,\\
Neil Malur, Ashwin Padaki, Nastia Polina,\\
Will Sun, Jacob Tan, Andrew The}
\date{}

\begin{document}

\maketitle

\begin{abstract}
We discuss a generalization of logic puzzles in which truth-tellers and liars are allowed to deviate from their pattern in case of one particular question: ``Are you guilty?''
\end{abstract}

\section{Introduction by Tanya Khovanova}

The goal of the PRIMES program at MIT is to help gifted high-schoolers conduct research in mathematics. The program started in 2011 and is extremely successful \cite{EGK}. For example, PRIMES-2014 students won the first two places in the Siemens competition and the first three places in the area of basic research in the Intel Science Talent Search. When we started PRIMES many people doubted that high-schoolers could do high-level research in mathematics. We proved those people wrong. 

How early can students start research in mathematics? Can it be done in kindergarten? Probably not. But middle-schoolers can do interesting things in mathematics. Maybe not full-fledged research, but at least we can prepare them for research.

In 2015 we decided to start a program for middle-schoolers, which I would lead. The goal was to train them for math competitions, to teach them to think mathematically, and to prepare them for future research. This experimental program is called PRIMES STEP. We decided not to make research the main focus of the program because we didn't want the students to be disappointed if they did not succeed. We decided to run it as a math club, where research is a side project. This way we can expand the time we invest in research if the kids are interested and it goes well, or shrink it if it's not working.

I chose logic to start our experiments in research. There are many fun logic books and puzzles. In most of them people are divided into truth-tellers who always tell the truth and liars who always lie. One of the most famous books is by Raymond Smullyan \cite{Smullyan}, \textit{What is the Name of this Book?}, where truth-tellers are called knights and liars are called knaves. 

In reality you will never meet people who are so obsessive. I wanted to have a particular example where people do not tell the truth all the time. To simplify things, I decided that I would allow our truth-tellers to lie about only one question: ``Are you guilty?''

I didn't know where it would take us, but the journey was both unexpected and fun. I threw the problem at the students and they came up with a lot of ideas. I refrained from thinking about the project outside of class, leaving the students to lead the discoveries. 

We decided that we wanted to expand our study to four types of people: absolute truth-tellers, partial truth-tellers, absolute liars, and responsible liars. The kids formulated and proved some results. While doing this, they invented a number of fun logic puzzles. 

You can read about our discoveries in the following story, which we present in chronological order to show our thought processes.

\section{Truth-Tellers Island}

We are on an island---called the \textit{Truth-Tellers Island}---where everyone always tells the truth. That is, almost everyone and almost always. Some people here do indeed always tell the truth to the point of being obsessive. We call them the \textit{Absolute Truth-Tellers}. Some people tell the truth almost always. They lie only in one specific situation. If they are asked, ``Are you guilty?'' they lie only if they are indeed guilty. We call them the \textit{Partial Truth-Tellers}.

This is not the only question the partial truth-tellers lie about. They lie to any question that is equivalent to the above. For example, they would lie to questions such as, ``Did you commit the crime?'' or ``Did you steal the chocolate?'' In addition, if they volunteer information on the subject of them committing the crime, the same rule applies: they lie.

As we will discuss later, this situation is very non-trivial as they might be asked questions that are not equivalent to the question of guilt, but still imply that they might be guilty. But for the sake of our puzzles we assume that the islanders are not smart enough to analyze the implications and would tell the truth to any indirect question.

Here is one of the cases of the local island investigator, Detective Khovanova:

\textbf{Ashwin's Puzzle}

\begin{quote}

Detective Khovanova is investigating a theft of an Indian jewel. She knows that the perpetrators are one or more people from this list of five suspects: Ezra, Leon, Jacob, Andrew, and Will. She asks each of them if they committed the crime. These are their answers:

Ezra: ``I didn't commit the crime.'' \\
Leon: ``Ezra is correct. I committed the crime.''\\
Jacob: ``I am innocent. Will is as well.''\\
Andrew: ``I committed the crime with Leon.''\\
Will: ``I am innocent. Jacob is not.''

Who committed the crime? Can you identify the absolute truth-tellers?
\end{quote}

\textbf{Ashwin's Solution.} Ezra is not guilty because Leon confirms it. Leon is guilty because Andrew points to him. Jacob is guilty, because Will rats him out. Andrew is guilty by his own admission. Will is innocent by Jacob's statement. 

Leon and Andrew are absolute truth-tellers because they admitted their guilt. Jacob is not. We can't say if Ezra or Will are absolute truth-tellers or not.

We used the following observation to solve this puzzle:

\begin{lemma}\label{thm:1}
On Truth-Tellers Island if someone says they committed the crime, then they actually committed the crime. In addition, statements about other people are always true.
\end{lemma}

In the following puzzle the crime can be solved because people say where and with whom they spent the day. Here we assume that if people were together, then they are either  guilty together or innocent together.

\textbf{Jacob's Puzzle}

\begin{quote}
Four people are all suspects of breaking into the president's Glorious Bank sometime during the day yesterday, but only two out of the four suspects actually broke into the president's Glorious Bank. The following statements were taken from each person:

Jonathan: ``I was guarding the palace all night yesterday. Ezra was guarding with me.''\\
Ashwin: ``I did not break in. I was with Ezra all day yesterday.''\\
Ezra: ``I did not break in. But Jonathan was acting strange all day yesterday.''\\
Andrew: ``I did not break in. I saw Jonathan guarding the palace all day long yesterday.''

Can Detective Khovanova figure out who broke into the president's Glorious Bank?
\end{quote}

\textbf{Jacob's Solution.} Jonathan's statement is misleading because it is about night not day. The information he provides is irrelevant. Ashwin was with Ezra. But Jonathan didn't break in as he was guarding. So it was Ashwin and Ezra.

Jacob's puzzle leads us to the following lemma, where we assume that if people committed the crime together, then they knew each other.

\begin{lemma}\label{thm:2}
On Truth-Tellers Island if a crime was committed by at least two people together who didn't hide each other's identities from each other, the detective can always solve the crime.
\end{lemma}

\begin{proof}
The detective asks each person about the guilt of all the other suspects. Since Truth-Teller islanders do not lie about other people, each criminal will be identified by his partners.
\end{proof}

The important part in the above proof is not that the criminals were together, but that every criminal's actions are known to someone else. The same method of solving crimes works for the following situation.

\begin{lemma}\label{thm:3}
On Truth-Tellers Island, if for every criminal there is at least one other person who knows that this criminal committed the crime, then the crime can be solved.
\end{lemma}

Now Detective Khovanova takes on a case where no one other than the criminals themselves knows who committed the crime.

\textbf{Neil's Puzzle}

\begin{quote}
The mayor's favorite cuff-links were stolen and Detective Khovanova is trying to figure out who did it. She knows from her previous interrogations that the crime was committed by exactly one person out of four prime suspects. The suspects have never met or heard about each other. In response to the detective's questions, the suspects answer truthfully as much as they can to the extent of their knowledge. Detective Khovanova was able to figure out who did it by asking them the same question. What is the question?
\end{quote}

\textbf{Neil's Solution.} The question could be, ``Did I do it?'' The innocent person would answer: ``I do not know.'' The guilty person would answer, ``No.''

This is exactly the situation in which people do not admit their guilt directly, but it is deduced from their answers. 

We just covered the case in which one person committed the crime and no one else knows who the guilty party is. Earlier we covered the situation in which the guilt of every criminal was known by another person. 

What remains to discuss are those weird cases in which the criminals committed the crime together but did not know each other. This could happen, for example, if the criminals found each other online where they used pseudonyms and they robbed a bank together while wearing masks. What should Detective Khovanova do in this case?

The interesting lemma below shows that the detective can still solve the crime if she knows exactly how many people committed the crime. 

\begin{lemma}
Suppose on Truth-Tellers Island the crime was committed such that no one knows about anyone else who committed the crime. Even the criminals do not know any of their partners. However, if the number of people who committed the crime is public knowledge, then the detective can find out who committed the crime.
\end{lemma}

\begin{proof}
Suppose $m$ people committed the crime. Then the detective picks person X and gives X all possible lists of $m$ people not including X and asks if it is possible that any set from this list of people committed the crime. This is equivalent to asking, ``Are there $m$ people other than you who could have committed the crime.'' The guilty people know that they are one of the gang, so there could not be other $m$ people who committed the crime. So they would say `no.' The innocent people would say `yes.'
\end{proof}

Is there a situation in which the detective can't identify the guilty party? Is it possible that some crimes remain unsolved? The only case we have not yet discussed is when the criminals didn't know each other at all; only the criminals themselves knew that they committed the crime; and the total number of criminals is unknown. 

Just yesterday such a crime occurred. During a parade the vehicle transporting all the cash broke down and money flew everywhere. Some people picked up the money but no one knows who. In addition, no one knows how many people stole the money. Can this crime be solved? 

\begin{lemma}
The crime where no one knows how many people committed the crime and no one has any information about anyone else can be solved.
\end{lemma}

\begin{proof}
The detective asks, ``Is it possible that everyone committed the crime?'' The innocent persons know of their innocence so they would say ``no.'' The guilty people would say ``yes.''
\end{proof}

We can combine all our lemmas into one big theorem:

\begin{theorem}\label{thm:tt}
On Truth-Tellers Island it is always possible to solve any crime.
\end{theorem}

\begin{proof}
First the detective asks everyone about everyone else. This way she finds all the criminals who failed to keep their crimes hidden from other people. There may still be undiscovered criminals who no one but themselves knows about. Then she asks everyone who has not yet been identified as a criminal, ``Did someone else commit the crime?'' If the answer is `no' or `I don't know,' the person is innocent. If the answer is `yes,' the person is guilty.
\end{proof}

The next puzzle shows how to use partial information about the crime to create a misdirected question.

\textbf{Leon's Puzzle}

\begin{quote}
On Truth-Tellers Island, Jolly Ranchers are worth 1,000,000 gold coins. Ezra previously collected Jolly Ranchers secretly in a safe. One day, Ezra realized his blue Jolly Rancher was stolen from his safe, while the other colors were not touched. When the police arrived, he accused his five neighbors, whom the police immediately seized for questioning. What one question can Detective Khovanova ask in order to determine who is innocent
and who is guilty, assuming that the thief has not confided in anyone else?
\end{quote}

\textbf{Leon's Solution.} The question is, ``What color was the stolen Jolly Rancher?''

In this puzzle the detective uses the information that only the criminal could know, in order to figure out who committed the crime. For every crime, there is always such information, other than peoples' guilt, to differentiate between the innocent and the criminal. For example, in one of the previous cases, when people stole money they found lying in the street, the detective could have asked:
\begin{itemize}
\item How much money did you take?
\item Were there \$100 bills?
\end{itemize}

Moreover, the detective can ask people what they did during the crime, instead of asking about the suspect's guilt. For example:

\begin{itemize}
\item Did you hold a knife in your hands at noon yesterday?
\item How much sleep did you get last night?
\end{itemize}

This provides another way of solving crimes on Truth-Tellers Island.

If partial truth-tellers are smart and do not want to get caught, they need to expand the number of questions they lie about. But if our mathematical model allows for more than one lie, the model gets more complicated. We started this research by thinking that absolute Truth-Tellers are too far from reality: normal people lie from time to time. In order to reflect real life more closely, we decided to introduce a simple variation from absolute Truth-Tellers, allowing them to lie to one question and only in one particular situation. We just discovered that the detective can by-pass this particular question by using alternative questions, of which there are a whole variety.

Our discussion also shows that in real life when people start lying they are forced to continue lying to hide their participation in the crime. One lie leads to a chain of lies.

To make the rest of our puzzles fun, let's introduce another assumption for the rest of the paper: The detective cannot ask questions about the colors of candy or the types of bills. She's only allowed to ask about participation in the crime.

Andrew invented a more convoluted puzzle, which highlights an issue we have yet to address. In his puzzle he used compound sentences. How do the islanders treat compound sentences where one of the compounds is ``I am not guilty''? The partial truth-tellers who are not guilty as well as the absolute truth-tellers always tell the truth anyway. So their compound sentences are also always true. The partial truth-tellers who are guilty only lie about their guilt. That means, they treat the sentence ``I am not guilty'' as the truth. If the statement of their guilt is a part of a compound sentence, we assume that the whole sentence is true on the condition that they are not guilty.

\textbf{Andrew's Puzzle}

\begin{quote}
A gold star was stolen from the bank. Detective Khovanova knows that all her suspects are partial truth-tellers. All the perpetrators are among the three suspects. These are their statements:

Jakob: ``If Essra did it, then he wasn't alone.'' \\
Essra: ``I did it or Jakob didn't do it.'' \\
Johnatan: ``If I didn't do it, then Essra did it.''

Who stole the gold star?
\end{quote}

\textbf{Andrew's Solution.}

Partial truth-tellers behave like absolute truth-tellers who are innocent. So Essra's admission of guilt should be considered a lie. But his whole compound sentence must be true. Therefore we can conclude that Jakob didn't do it. Johnatan's premise that he didn't do it must be true. It follows that Essra did it. From Jakob's statement it follows that Essra didn't do it alone; that means Johnatan did it too. The criminals are Essra and Johnatan.

\section{Liars Island}

Off the coast of Truth-Tellers Island, there is a smaller island, called \textit{Liars Island}, where all the liars are sent to live. Here people almost always lie. Like on Truth-Tellers Island, there are two types of people on Liars Island. People who always lie are called the \textit{Absolute Liars}. Other people lie most of the time. They only tell the truth when they are guilty and asked if they are guilty. We call them the \textit{Responsible Liars}. These two
types of people make up the entire population of this island.

Detective Khovanova was invited to Liars Island to solve its crimes.

\textbf{Jonathan's Puzzle}

\begin{quote}
Someone stole Ben's special snowflake. There are three suspects: Mike, Leon, and Ashwin. No one else could have committed the crime. Here are the statements:

Mike: ``I didn't commit the crime.'' \\
Leon: ``Two people committed the crime.''\\
Ashwin: ``Both Leon and Mike participated in the crime.''

Who committed the crime?
\end{quote}

\textbf{Jonathan's Solution.} Mike says that he didn't commit the crime, which means that he did. Leon's statement is a lie. Therefore, either one person committed the crime---in this case this would have to be Mike, or all three of them did it. Ashwin's statement is a lie. Therefore, all three people couldn't have done it together. The crime was committed by Mike.

Can we transfer the lemmas and theorems from Truth-Tellers Island to Liars Island by symmetry? It might be tempting to think so, but the answer is ``no.'' We can only transfer statements where the questions are of the ``yes-or-no'' type. Such questions give us the same information. Open-ended questions might not work. For example, when we ask a liar what kind of bills were there at the crime scene, if the liar didn't see the bills s/he can say any number. 

Lemma~\ref{thm:1} uses only yes-or-no questions in the proof, so it can be extended to Liars Island by symmetry:

\begin{lemma}
On Liars Island if someone says they are not guilty, then they are guilty. In addition, statements about other people are false.
\end{lemma}

Similarly we can combine Lemma~\ref{thm:2} and Lemma~\ref{thm:3} into the following statement:

\begin{lemma}
On Liars Island if for every criminal there is at least one other person who knows that this criminal committed the crime, then the crime can be solved.
\end{lemma}

Let us move on to the case where no one knows who committed the crime. Consider Neil's puzzle where exactly one person stole the mayor's favorite cuff-links and no one knows who.

\textbf{Solution to Neil's puzzle on Liars Island.} On Truth-Tellers Island we suggested a question for Detective Khovanova, ``Did I do it?'' This question worked there. What happens if we use the same question on Liars Island? A guilty person would know that the detective didn't do it, so s/he would say ``Yes.'' An innocent person doesn't know. What would s/he say? We are in a tricky territory that few logic books discuss. It is better to stick to yes-or-no questions. In Neil's puzzle, the question might be, ``Is it possible that I committed the crime?''

We can adjust the proof of the main theorem for Truth-Tellers Island (Theorem~\ref{thm:tt}) to make it work with only yes-or-no questions:

\begin{theorem}\label{thm:l}
On Liars Island it is always possible to solve any crime.
\end{theorem}

\begin{proof}
We just need to convert the question ``Did someone else commit the crime?'' which we used on Truth-Tellers Island into a yes-or-no question. For example, the detective can ask a person ``I wrote down a random list of people not including you. Is it possible that they are the only ones who committed the crime?'' This question can be used when the number of criminals is known and also when it is unknown. If the person under questioning is guilty, the answer is ``yes.'' If the person is not guilty, the answer is ``no.''
\end{proof}

Here is a complicated puzzle related to a crime on Liars Island:

\textbf{Ezra's Puzzle} 

\begin{quote}
One day a number of government files were stolen. Detective Khovanova reduced the pool of the suspects to three people: Andrew, Neil, and Ben. When Detective Khovanova interviews the three people she gets the following responses.

Andrew: ``I didn't commit the crime. Ben will lie when questioned about his guilt.''\\
Ben: ``I committed the crime. Neil was also involved.''\\
Neil: ``I am totally innocent. My statement isn't necessary to solve the case. It was Ben who committed the crime by himself.''

Who committed the crime?
\end{quote}

\textbf{Ezra's Solution.}
Andrew's second statement must be a lie. This means that Ben is a responsible liar and tells the truth when questioned about his guilt, if he is actually guilty. When Ben says that he committed the crime he is telling the truth. His second statement must be a lie, so Neil was not involved. Since Neil says that Ben committed the crime by himself, the opposite must be true. Since Ben didn't commit the crime by himself, he must have been working with Andrew. Therefore, it was Andrew and Ben who committed the crime.

\section{Mixing People}

People from both islands went to the mainland to celebrate New Year's together. Unfortunately, crimes happen even during festive times.

\textbf{Ben's Puzzle}

\begin{quote}
Someone has dug a hole through the Trump Wall and stolen the holy potato of Joe Biden. Detective Khovanova must solve this case to get a beef burrito offered by the great and almighty Lordship. There are four suspects, one of each type. Exactly two of them did it. Who dug a hole through the Trump Wall and what is the type of each person?

Neil: ``I'm a partial truth-teller.  Mike is not.''\\
Mike: ``I'm a partial truth-teller.  Neil is not.  He is a scrub though.''\\
Nastia: ``I'm a banana.  Neil is telling the truth.'' \\
Leon: ``I like potatoes.  Neil is an absolute liar.''\\
Mike: ``I have no idea who did it.''\\
Neil: ``I didn't do it. And Leon didn't do it either.''\\
Nastia: ``I did it.  Leon did not.''\\
Leon: ``I did it.''
\end{quote}

\textbf{Ben's Solution.} Nastia is lying because she cannot be a banana. So Neil is lying as well because Nastia calls Neil a truth-teller.  So Nastia and Neil are liars and Mike and Leon are truth-tellers.  Leon must have done it since he is a truth-teller and he says ``I did it.''  It also follows that Leon is an absolute truth-teller and Mike is a partial truth-teller. We know that Neil is an absolute liar because Leon, a truth-teller, says so. Neil's ``I didn't do it'' is a lie, so Neil and Leon did it together. By elimination Nastia is a responsible liar.

To solve any crime conducted by people from both islands Detective Khovanova just needs to separate truth-tellers from liars. It is easy to do. She can ask any question she knows the answer to. For example, she can ask how much is two plus two, or what color is the sun. Therefore, we can extend our main theorems:

\begin{theorem}\label{thm:main}
If a crime was committed by people from both of these islands it is always possible to solve it.
\end{theorem}

Let's try another puzzle.

\textbf{Mike's Puzzle}

\begin{quote}
One day, four people are taken in as suspects for stealing an ancient scroll from a museum; two of them are truth-tellers and two are liars. Their names are Andrew, Ben, Jacob, and Jonathan. Detective Khovanova interrogates them. The detective also knows that only one person stole the scroll and that he is the only absolute liar among the suspects. In separate rooms, she asks each of the suspects, ``Did you steal the ancient scroll?''

Andrew says: ``It was not me. It was Ben.''\\
Ben says: ``I did not do it.''\\
Jacob says: ``Ben or Jonathan is an absolute liar.''\\
Jonathan says: ``It was Ben. Jacob is a liar.''

Is it possible to figure out who stole the scroll? If so, who was it? 
\end{quote}

\textbf{Mike's Solution.} Suppose Ben committed the crime. Then we have three people, Andrew, Jacob, and Jonathan making true statements not about their own guilt. This is a contradiction. This means Ben is innocent. Therefore, Andrew and Jonathan must be liars. In addition, Jacob is a truth-teller and, thus, either Ben or Jonathan is an absolute liar who committed the crime. We know it is not Ben. Therefore, the crime was committed by Jonathan.

Nastia's puzzle mixes people but restricts the number of types of people to two. 

\textbf{Nastia's puzzle.}

\begin{quote}
The Rube Cube Club only lets in two types of people, but we do not know which types. When a golden cue stick is stolen from the Club, Detective Khovanova figures it was exactly one of three members of the club: Neil, Leon, or Ben. Detective Khovanova politely listens as the three suspects talk to her.

Neil: ``I am guilty. Leon is not guilty.''\\
Leon: ``I am guilty. Ben is not guilty. Neil is from Liars Island.''\\
Ben: ``I am guilty. Neil is an absolute liar.''\\
Neil: ``Ben is from Truth-Tellers Island.''

Who is guilty, and what are the two types of people who are allowed in the Rube Cube?
\end{quote}

\textbf{Nastia's Solution.} If we assume that Leon is a liar, then Neil is a truth-teller. 
According to Neil's last statement, Ben is a truth-teller too. But Ben says that Neil is a liar. This is a contradiction. 

Therefore, Leon is a truth-teller. Then according to Leon's statement, Ben is not guilty, and Neil is a liar. Thus, according to Neil's first statement, Leon is guilty. Also, because Neil lied in his second statement, Ben is a liar. As Leon admitted his guilt, he is an absolute truth-teller. Finally, as Ben's statement identifying Neil as an absolute liar is a lie and Neil is from Liars Island, Neil must be a partial liar, and therefore Ben is also a partial liar.

Leon committed the crime and the two types of people that are allowed in Rube Cube are absolute truth-tellers and partial liars.

Similar to Nastia's puzzle there are only two types of people in Will's puzzle.

\textbf{Will's puzzle.}

\begin{quote}
One night at a party, where partial truth-tellers and responsible liars came together, Johnathan's pet dinosaur was stolen. There were three suspects: Andrew, Ezra, and Jacob. Later when they were questioned, here was the conversation:

Andrew: ``Ezra did it. Ezra did it alone.''\\
Ezra: ``I stole Johnathan's pet dinosaur. Jacob is not guilty.''\\
Jacob: ``Exactly two people committed the crime. I had nothing to do with this case.''

Who stole Johnathan's pet dinosaur? And who is who?
\end{quote}

\textbf{Will's Solution.} Right off the bat, we see that Ezra is a responsible liar and Jacob is a partial truth-teller. This is because responsible liars say that they are guilty independently of whether they are guilty or not and partial truth-tellers always say that they are innocent. It follows that two people committed the crime and Jacob is one of them. As Andrew accuses only Ezra, we can infer that Andrew is a responsible liar and Ezra is innocent. The pet dinosaur was stolen by Andrew and Jacob.

\section{Acknowledgements}
We would like to thank the PRIMES STEP program for the opportunity to do this research. In addition, we are grateful to PRIMES STEP Director, Dr.~Slava Gerovitch, for his help and support.

\end{document}